\documentclass[10pt]{article}
\usepackage{amsmath,amssymb,amsthm}
\usepackage{epsfig}
\usepackage{color}
\usepackage{tikz}

\usepackage{tikz,epsf}

\usepackage{amsfonts}
\usepackage{amscd}
\usepackage{amsmath}
\usepackage{graphicx}
\usepackage{caption,subcaption}
\title{On the inverse signed total domination number in graphs}

\author {
D.A. Mojdeh\thanks{Corresponding author}\\
Department of Mathematics\\
University of Mazandaran\\
Babolsar, IRI\\
{\tt damojdeh@umz.ac.ir}\vspace{3mm}\\
Babak Samadi\\
Department of Mathematics\\
Arak University\\
Arak, IRI\\
{\tt b-samadi@araku.ac.ir}
}

\date{}

 \newtheorem{theorem}{Theorem}[section]

\newtheorem{lemma}[theorem]{Lemma}
\newtheorem{observation}[theorem]{Observation}
\newtheorem{proposition}[theorem]{Proposition}

\theoremstyle{definition}

\begin{document}

\maketitle
\begin{abstract}
\noindent In this paper, we study the inverse signed total domination number in graphs and present new sharp lower and upper bounds on this parameter. For example by making use of the classic theorem of Tur\'{a}n (1941), we present a sharp upper bound on $K‎_{r+1}‎$-free graphs for $r‎\geq2‎$. Also, we bound this parameter for a tree from below in terms of its order and the number of leaves and characterize all trees attaining this bound.
\end{abstract}
\vspace{1mm}
\section{Introduction}

Throughout this paper, let $G$ be a finite connected graph with vertex set $V=V(G)$, edge set $E=E(G)$, minimum degree $\delta=\delta(G)$ and maximum degree $\Delta=\Delta(G)$. For any vertex $v\in V$, $N(v)=\{u\in G\mid uv\in E\}$ denotes the {\em open neighborhood} of $v$ in $G$, and $N[v]=N(v)\cup \{v\}$ denotes its {\em closed neighborhood}. For all $A,B‎\subseteq V‎$ let $[A,B]$ be the set of edges having one end point in $A$ and the other in $B$. We use \cite{we} as a reference for terminology and notation which are not defined here.\\
A set  $D\subseteq V$ is a {\em total dominating set} in $G$ if each vertex in $V$ is adjacent to at least one vertex in $D$. The {\em total domination number $\gamma_{t}(G)$} is the minimum cardinality of a total dominating set in $G$.\\
The {\em $k$-tuple total dominating set} (or {\em $k$-total dominating set}) was introduced by Kulli \cite{k}, as a subset $D\subseteq V$ with $|N(v)\cap D|\geq k$ for all $v\in V$, where $1\leq k\leq \delta$. The {\em $k$-tuple total domination number} (or {\em $k$-total domination number}), denoted $\gamma_{\times k,t}(G)$, is the smallest number of vertices in a $k$-tuple total dominating set. Clearly, $\gamma_{\times 1,t}(G)=‎\gamma‎_{t}(G)‎‎$.\\
Let $B\subseteq V$. For a real-valued function $f:V\rightarrow R$ we define $f(B)=\sum_{v \in B}f(v)$. Also, $f(V)$ is the weight of $f$. A {\em signed total dominating function}, abbreviated STDF, of $G$ is defined in \cite{z} as a function $f:V\rightarrow \{-1,1\}$ such that $f(N(v))\geq1$ for every $v\in V$. The {\em signed total domination number} (STDN) of $G$, $\gamma_{st}(G)$, is the minimum weight of a STDF of $G$. If we replace $"‎\geq‎"$ and "minimum" with $"‎\leq‎"$ and "maximum", respectively, in the definition of STDN, we will have the {\em signed total $2$-independence function} (ST$2$IF) and the {\em signed total $2$-independence number} (ST$2$IN) of the graph $G$. This concept was introduced in \cite{ws} and studied in \cite{w,w1} as the {\em negative decision number}.\\
An {\em inverse signed total dominating function}, abbreviated ISTDF, of $G$ is defined in \cite{hfx} as a function $f:V\rightarrow \{-1,1\}$ such that $f(N(v))‎\leq0‎$ for every $v\in V$. The {\em inverse signed total domination number} (ISTDN), denoted by $‎\gamma‎‎^{0}_{st}‎‎(G)‎$, is the maximum weight of an ISTDF of $G$. For more information the reader can consult \cite{{ansv}}.\\
In this paper, we continue the study of the concept of inverse signed total domination in graphs. In Section $2$, we present a sharp upper bound on a general graph $G$ by considering the concept of $k$-tuple total domination in graphs. Moreover, as an application of the well-known theorem of Tur\'{a}n \cite{t} about $K_p$-free graphs we give an upper bound on $\gamma‎‎^{0}_{st}‎‎(G)$ for a $K‎_{r+1}‎$-free graph $G$ and show the bound is the best possible by constructing an $r$-partite graph attaining the bound. In Section $3$, we discuss ISTDN for regular graphs and give sharp lower and upper bounds on the ISTDN of regular graphs. Finally, in the next section we show that this parameter is not bounded from both above and below in general, even for trees. Also, we give a lower bound on the ISTDN of a tree $T$ as $\gamma‎‎^{0}_{st}‎‎(T)‎\geq-n+2(‎\lfloor‎\frac{\ell‎_{1}‎}{2}‎‎\rfloor‎‎+...+‎\lfloor‎\frac{\ell‎_{s}‎}{2}‎‎\rfloor)‎$, where $\ell‎_{1},...,\ell‎_{s}$ are the number of leaves adjacent to its $s$ support vertices, and characterize all trees attaining this bound.


\section{Two upper bounds}

Throughout this paper, if $f$ is an ISTDF of $G$, then we let $P$ and $M$ denote the sets of those vertices which are assigned $1$ and $-1$ under $f$, respectively. We apply the concept of tuple total domination to obtain a sharp upper bound on $\gamma_{st}(G)$. 
It is easy to check the proof of the following observation.
\begin{observation}
 Let $K_n$ and $C_n$ be complete graph  and cycle with $n$ vertices. Then\\

\emph{(i)}\ $\gamma‎‎^{0}_{st}‎‎(K_n) = \begin{cases} -2\ \ &\text{if}\ \ n‎\equiv0\ (mod\ 2)‎\\
                              -1\ \  &\text{if}\ \ n‎\equiv1\ (mod\ 2)‎.

                    \end{cases}$\\

\emph{(ii)}\ $\gamma‎‎^{0}_{st}‎‎(C_n) = \begin{cases} \ 0\ \ &\text{if}\ \ n‎\equiv0\ (mod\ 4) \\
                                -1\ \ &\text{if}\ \ n‎\equiv1\ or\ 3\ (mod\ 4)\\
-2\ \ &\text{if}\ \ n‎\equiv2\ (mod\ 4).

                    \end{cases}$\\

\emph{(iii)}\  $‎\gamma‎_{t}‎‎(K_{n})=2$ for $n‎\geq 2‎$. \\

\emph{(iv)}\  $\gamma‎‎_{t}‎‎(C_n) = \begin{cases} \left\lceil \frac{n}{2}\right\rceil +1\ &\text{if}\ \ n‎\equiv2\ (mod\ 4)\\
                                \left\lceil \frac{n}{2}\right\rceil &\text{otherwise}.

                    \end{cases}$\\
  \end{observation}
We make use of them to show that the following bound is sharp.

\begin{theorem}
If $G$ is a graph of order $n$ and minimum degree $‎\delta‎\geq1‎‎$, then
$$\gamma‎‎^{0}_{st}‎‎(G)‎\leq n-2‎\lceil‎\frac{2‎\gamma‎_{t}(G)‎‎+‎\delta-2‎}{2}‎‎\rceil‎‎‎$$
and this bound is sharp.
\end{theorem}
\begin{proof}
Let $f:V‎\rightarrow\{-1,1\}‎$ be a maximum ISTDF of $G$. The condition $f(N(v))‎\leq0‎$, for all $v‎\in V‎$, follows that $|N(v)‎\cap M‎|‎\geq ‎\lceil‎\frac{deg(v)}{2}‎‎\rceil‎\geq‎‎‎ ‎\lceil‎‎\frac{‎\delta‎}{2}‎\rceil‎‎‎$. So, $M$ is a $\lceil‎‎\frac{‎\delta‎}{2}‎\rceil$-tuple total dominating set in $G$ and therefore
\begin{equation}
(n-\gamma‎‎^{0}_{st}‎‎(G))/2=|M|‎\geq ‎\gamma‎_{‎\times\lceil‎‎\frac{‎\delta‎}{2}‎\rceil,t}‎‎‎(G).
\end{equation}
Now let $D$ be a minimum $\lceil‎‎\frac{‎\delta‎}{2}‎\rceil$-tuple total dominating set in $G$. Hence, $|N(v)‎\cap D|‎\geq ‎\lceil‎‎\frac{‎\delta‎}{2}‎\rceil$, for all $v\in V$. Let $u\in D$. Then $|N(v)‎\cap (D‎\setminus\{u\}‎)‎|‎\geq ‎\lceil‎‎\frac{‎\delta‎}{2}‎\rceil-1$, for all $v\in V$. This shows that  $D‎\setminus\{u\}$ is a ($‎\lceil‎‎\frac{‎\delta‎}{2}‎\rceil-1$)-tuple total dominating set in $G$. Therefore, $‎\gamma‎_{‎\times\lceil‎‎\frac{‎\delta‎}{2}‎\rceil,t}‎‎‎(G)‎\geq ‎\gamma‎_{‎\times(\lceil‎‎\frac{‎\delta‎}{2}‎\rceil-1),t}‎‎‎(G)+1‎$.
If we iterate this process we finally arrive at
$$\gamma‎_{‎\times\lceil‎‎\frac{‎\delta‎}{2}‎\rceil,t}‎‎‎(G)‎\geq ‎\gamma‎_{‎\times(\lceil‎‎\frac{‎\delta‎}{2}‎\rceil-1),t}‎‎‎(G)+1‎\geq...‎\geq \gamma‎_{‎\times1,t}‎‎‎(G)‎‎+\lceil‎‎\frac{‎\delta‎}{2}‎\rceil-1=\gamma‎_{‎t}‎‎‎(G)‎‎+\lceil‎‎\frac{‎\delta‎}{2}‎\rceil-1.$$
Thus, (1) yields
$$(n-\gamma‎‎^{0}_{st}‎‎(G))/2‎\geq \gamma‎_{‎t}‎‎‎(G)‎‎+\lceil‎‎\frac{‎\delta‎}{2}‎\rceil-1‎$$
as desired. Moreover, this bound is sharp. It is sufficient to consider the complete graph $K‎_{n}‎$, when $n‎\geq2‎$ and the cycle $C‎_{n}‎$, when $n‎\geq3‎$.
\end{proof}

Note that the difference between $\gamma‎‎^{0}_{st}‎‎(G)$ and $n-2‎\lceil‎\frac{2‎\gamma‎_{t}(G)‎‎+‎\delta-2‎}{2}‎‎\rceil‎‎‎$ may be large. It is easy to check that for complete bipartite graph $K_{m,n}$;\vspace{0.25mm}\\
$\gamma‎‎^{0}_{st}‎‎(K_{m,n}) = \begin{cases} \ 0\ \ &\text{if}\ \ n,m‎\equiv0\ (mod\ 2) \\
                                -1\ \ &\text{if\ $m$\ and\ $n$\ have\ different\ parity} \\
-2\ \ &\text{if}\ \ n,m‎\equiv1\ (mod\ 2).

                    \end{cases}$\vspace{0.75mm}\\
While $n+m-2‎\lceil‎\frac{2‎\gamma‎_{t}(K_{m,n})‎‎+‎\delta-2‎}{2}‎‎\rceil\ge n-3‎‎‎$ where $n=max\{m,n\}$. The upper bound in Theorem \ref{th7} works better for bipartite graphs. We first recall that a graph is $K_p$-free if it does not contain the complete graph $K_p$ as an induced subgraph. For our next
upper bound, we use the following well-known theorem of Tur\'{a}n \cite{t}.
\begin{theorem}\label{th6}
If $G$ is a $K_{r+1}$-free graph of order $n$, then
$$|E(G)|\le \frac{r-1}{2r}\cdot n^2.$$
\end{theorem}
\begin{theorem}\label{th7}
Let $r‎\geq2‎$ be an integer, and $G$ be a $K‎_{r+1}‎$-free graph of order $n$. If $c=‎\lceil‎\frac{‎\delta‎}{2}‎‎\rceil‎‎$, then
$$\gamma‎‎^{0}_{st}‎‎(G)‎‎\leq n-\frac{r}{r-1}\left(-c+\sqrt{c^2+4\frac{r-1}{r}cn}\right)$$
and this bound is sharp.
\end{theorem}
\begin{proof}
Let $f$ be an ISTDF of $G$. Let $v\in P$. Since $f(N(v))‎\leq0‎$, then $|N(v)‎\cap M‎|‎\geq \lceil‎‎\frac{‎\delta‎}{2}‎\rceil‎‎‎‎$. Therefore
\begin{equation}
|[M,P]|‎\geq ‎\lceil‎‎\frac{‎\delta‎}{2}‎\rceil‎‎‎‎|P|.
\end{equation}
Furthermore, Theorem \ref{th6} implies
$$|[M,P]|=\sum_{v\in M}|N(v)\cap P|\le \sum_{v\in M}|N(v)\cap M|=2|E(G[M])|\le\frac{r-1}{r}|M|^2.$$
Combining this inequality chain and (2), we arrive at
$$‎\frac{r-1}{r}|M|^2+c|M|‎-cn‎\geq0.‎$$
Solving the above inequality for $|M|$ we obtain
$$‎‎|M|‎\geq ‎\frac{r}{2(r-1)}\left(-c‎‎+‎\sqrt{c^2+4‎\frac{r-1}{r}cn‎}‎\right).$$
Because of $|M|=(n-\gamma‎‎^{0}_{st}‎‎(G)‎‎)/2$, we obtain the desired upper bound.\\
That the bound is sharp, can be seen by constructing an $r$-partite graph attaining this upper bound as follows. Let $H‎_{i}‎$ be a complete bipartite graph with vertex partite sets $X‎_{i}‎$ and $Y‎_{i}‎$, where $|X‎_{i}‎|=r-1$ and $|Y‎_{i}‎|=(r-1)^2$ for all $1‎\leq‎ i‎\leq r‎$. Let the graph $H(r)$ be the disjoint union of $H‎_{1},...,H‎_{r}‎‎$ by joining each vertex of $X‎_{i}‎$ (in $H‎_{i}‎$) with all vertices in union of $X_{j}‎$, $i‎\neq j‎$. Also, we add $(r-1)^3$ edges between $Y‎_{i}‎$ and the union of $Y‎_{j}‎$, $i‎\neq j‎$, so that every vertex of $Y‎_{i}‎$ has exactly $r-1$ neighbors in this union. Now let $Z‎_{i}=X‎_{i}‎\cup Y‎_{i+1}‎‎‎‎$, for all $1‎\leq‎ i‎\leq r‎‎‎$ (mod $r$). Then $H(r)$ is an $r$-partite graph of order $n=r^2(r-1)$ with partite sets $Z‎_{1},...,Z‎_{r}‎$. Clearly, every vertex in $Y‎_{i}‎$ has the minimum degree $‎\delta=2r-2‎$ and hence $c=r-1$. Now we define $f:V(H(r))‎\rightarrow\{-1,1\}‎$, by
$$f(v)=\left \{
\begin{array}{lll}
-1 & \mbox{if} & v\in X_{1}‎\cup...‎\cup X‎_{r}‎‎‎‎ \\
1 & \mbox{if} & v\in Y_{1}‎\cup...‎\cup Y‎_{r}.
\end{array}
\right.$$
It is easy to check that $f$ is an ISTDF of $H(r)$ with weight
$$r(r-1)^2-r(r-1)=n-\frac{r}{r-1}\left(-c+\sqrt{c^2+4\frac{r-1}{r}cn}\right).$$
This completes the proof.
\end{proof}


\section{Regular graphs}

Our aim in this section is to give sharp lower and upper bounds on the ISTDN of a regular graph. Henning \cite{h} and Wang \cite{w} proved that for an $r$-regular graph $G$,
\begin{equation}
\gamma‎‎_{st}‎‎(G)‎\leq‎\left \{
\begin{array}{lll}
‎(\frac{r^2+r+2}{r^2+3r-2})n‎ & \mbox{if} & r‎\equiv0\ (mod\ 2)‎ ‎‎‎‎ \vspace{1.5mm}\\
‎(\frac{r^2+1}{r^2+2r-1})n‎ & \mbox{if} & r‎\equiv1\ (mod\ 2)
\end{array}
\right.
\end{equation}
(see \cite{h}), and
\begin{equation}
\alpha_{st}^{2}(G)‎\geq‎\left \{
\begin{array}{lll}
‎(\frac{1-r}{1+r})n‎ & \mbox{if} & r‎\equiv0\ (mod\ 2)‎ ‎‎‎‎ \vspace{1.5mm}\\
‎(\frac{1+2r-r^2}{1+r^2})n‎ & \mbox{if} & r‎\equiv1\ (mod\ 2).
\end{array}
\right.
\end{equation}
(see \cite{w}). Furthermore, these bounds are sharp. Also, the following sharp lower an upper bounds on STDN and ST$2$IN of an $r$-regular graph $G$ can be found in \cite{z} and \cite{w1,ws}, respectively.
\begin{equation}
\gamma‎‎_{st}‎‎(G)‎‎\geq‎\left \{
\begin{array}{lll}
2n/r & \mbox{if} & r‎\equiv0\ (mod\ 2)‎ \\
‎n/r‎ & \mbox{if} & r‎\equiv1\ (mod\ 2)
\end{array}
\right.
\end{equation}
and
\begin{equation}
\alpha_{st}^{2}(G)‎‎\leq‎‎\left \{
\begin{array}{lll}
0& \mbox{if} & r‎\equiv0\ (mod\ 2)‎ ‎‎‎‎\\
n/r‎ & \mbox{if} & r‎\equiv1\ (mod\ 2).
\end{array}
\right.
\end{equation}

Applying the concept of tuple total domination we can show that there are special relationships among $\gamma‎‎^{0}_{st}‎‎(G)$, $\gamma_{st}‎‎(G)$ and $\alpha_{st}^{2}(G)$ when we restrict our discussion to the regular graph $G$.
\begin{theorem}\label{th8}
Let $G$ be an $r$-regular graph. If $r$ is odd, then $\gamma‎‎^{0}_{st}‎‎(G)=-\gamma_{st}‎‎(G)$ and if $r$ is even, then $\gamma‎‎^{0}_{st}‎‎(G)=\alpha_{st}^{2}(G)$.
\end{theorem}
\begin{proof}
By (1), we have
\begin{equation}
\gamma‎‎^{0}_{st}‎‎(G)‎\leq n-2‎\gamma‎_{‎\times‎‎\lceil‎\frac{r‎}{2}‎‎\rceil‎‎,t}(G)‎‎‎.
\end{equation}
Now let $D$ be a minimum $‎‎\lceil‎\frac{r‎}{2}‎‎\rceil‎$-tuple total dominating set in $G$. We define $f:V‎\rightarrow\{-1,1\}‎$, by
$$f(v)=\left \{
\begin{array}{lll}
-1 & \mbox{if} & v\in D \\
\ 1 & \mbox{if} & v\in V‎\setminus‎ D.
\end{array}
\right.$$
Taking into account that $D$ is a $‎‎\lceil‎\frac{r}{2}‎‎\rceil‎$-tuple total dominating set, we have $f(N(v))=|N(v)‎\cap(V‎\setminus D‎)‎|-|N(v)‎\cap D‎|=deg(v)-2|N(v)‎\cap D‎|‎\leq r-2‎\lceil‎\frac{‎r}{2}‎‎\rceil‎‎‎‎‎\leq0‎$. Thus, $f$ is an ISTDF of $G$ with weight $n-2‎\gamma‎_{‎\times‎‎\lceil‎\frac{r‎}{2}‎‎\rceil‎‎,t}(G)$ and therefore $\gamma‎‎^{0}_{st}‎‎(G)‎‎\geq‎ n-2‎\gamma‎_{‎\times‎‎\lceil‎\frac{r‎}{2}‎‎\rceil‎‎,t}(G)‎‎‎$. Now the inequality (7) implies
\begin{equation}
\gamma‎‎^{0}_{st}‎‎(G)‎‎=n-2‎\gamma‎_{‎\times‎‎\lceil‎\frac{r‎}{2}‎‎\rceil‎‎,t}(G).
\end{equation}
The following equalities for the STDN and the ST$2$IN of $G$ can be proved similarly.
\begin{equation}
\gamma‎‎_{st}‎‎(G)‎‎=2‎\gamma‎_{‎\times‎‎\lceil‎\frac{r+1‎}{2}‎‎\rceil‎‎,t}(G)-n
\end{equation}
and
\begin{equation}
\alpha_{st}^{2}(G)=n-2\gamma_{\times\lfloor\frac{r}{2}\rfloor,t}(G).
\end{equation}
From (8), (9) and (10), the desired results follow.
\end{proof}
Using Theorem \ref{th8} and inequalities (3),...,(6) we can bound $\gamma‎‎^{0}_{st}‎‎(G)‎‎‎$ of a regular graph $G$ from both above and below as follows.
\begin{theorem}
Let $G$ be an $r$-regular graph of order $n$. Then
$$(\frac{1-r}{1+r})n‎‎\leq ‎\gamma‎‎^{0}_{st}‎‎(G)‎‎‎\leq0,\ \ \  r‎\equiv0\ ‎(mod\ 2)$$
and
$$-‎(\frac{r^2+1}{r^2+2r-1})n‎\leq ‎\gamma‎‎^{0}_{st}‎‎(G)‎‎‎\leq-n/r,\ \ \  r‎\equiv1\ ‎(mod\ 2).$$
Furthermore, these bounds are sharp.
\end{theorem}
As an immediate result of Theorem \ref{th8} we have $\gamma‎‎^{0}_{st}‎‎(G)=-\gamma_{st}‎‎(G)$, for all cubic graph $G$. Hosseini Moghaddam et al. \cite{hmsv} showed that $‎\gamma‎_{st}(G)‎\leq 2n/3$ is a sharp upper bound for all connected cubic graph $G$ different from the Heawood graph $G‎_{14}‎$. Therefore, if $G$ is a connected cubic graph different from $G‎_{14}‎$, then $\gamma‎‎^{0}_{st}‎‎(G)‎\geq-2n/3‎$ is a sharp lower bound.
\begin{figure}
\begin{center}
\begin{tikzpicture}
\draw  (0:1.3cm) -- (25.7:1.3cm) -- (51.4:1.3cm) -- (77.1:1.3cm) --
(102.8:1.3cm) -- (128.5:1.3cm) -- (154.2:1.3cm) -- (179.9:1.3cm) --
(205.6:1.3cm) -- (231.3:1.3cm) -- (257:1.3cm) --
 (282.7:1.3cm) -- (308.4:1.3cm) --
(334.1:1.3cm)  -- cycle;
\draw  (0:1.3cm) -- (128.5:1.3cm) -- cycle;
\draw  (25.7:1.3cm) -- (257.5:1.3cm) -- cycle;
\draw  (51.4:1.3cm) -- (179.9:1.3cm) -- cycle;
\draw  (77.1:1.3cm) -- (308.4:1.3cm) -- cycle;
\draw  (102.8:1.3cm) -- (231.3:1.3cm) -- cycle;
\draw  (154.2:1.3cm) -- (282.7:1.3cm) -- cycle;
\draw  (205.6:1.3cm) -- (334.1:1.3cm) -- cycle;
\node at (0,-2) {Figure 1. The Heawood graph  $G_{14}$};
\foreach \x in {0,25.7,51.4,77.1,102.8,128.5,154.2,179.9,205.6,231.3,257,282.7,308.4,334.1}{
\draw [black,fill](\x:1.3cm) circle  (1.5pt);}
\end{tikzpicture}
\end{center}
\end{figure}

\section{Trees}

Henning \cite{h} proved that $‎\gamma‎_{st}(T)‎\geq2‎‎‎$, for any tree $T$ of order $n$. A similar result cannot be presented for $\gamma‎‎^{0}_{st}‎‎(T)$. In what follows we show that in general $\gamma‎‎^{0}_{st}‎‎(T)$ is not bounded from both above and below.  In fact, we prove a stronger result as follows.
\begin{proposition}
For any integer $k$, there exists a tree $T$ with $\gamma‎‎^{0}_{st}‎‎(T)=k$.
\end{proposition}
\begin{proof}
We consider three cases.\vspace{1mm}\\
{\bf Case 1.} Let $k=0$. Consider the $v‎_{1}-v‎_{2}-v‎_{3}-v‎_{4}‎‎‎‎$ path $P‎_{4}‎$. Clearly, $f(v‎_{1})=f(v‎_{4})=1$ and $f(v‎_{2})=f(v‎_{3})=-1$ define a maximum ISTDF of $P‎_{4}‎$ with weight $0$.
\vspace{1mm}\\
{\bf Case 2.} Let $k‎\geq1‎$. Consider the path $P‎_{k+4}‎$ on vertices $v‎_{1},...,v‎_{k+4}‎‎$, respectively, in which $v‎_{1}‎$ and $v‎_{k+4}‎$ are the leaves. Let $T$ be a tree obtained from $P‎_{k+4}‎$ by adding two leaves to each vertex $v‎_{i}‎‎$, for all $3‎\leq i‎\leq k+2‎‎$. The condition $f(N(v))‎\leq0‎$, for all ISTDF $f$ and $v\in V(T)$, follows that all support vertices must be assigned $-1$ under $f$. This shows that $f(v‎_{2}‎)=...=f(v‎_{k-3}‎)=-1$ and $f(v)=1$, for $v‎\neq ‎v‎_{2},...,v‎_{k-3}‎$, is a maximum ISTDF of $T$ with weight $k$.
\vspace{1mm}\\
{\bf Case 3.} Let $k‎\leq-1‎$. Let $T‎_{i}‎$ be a tree obtained from a the path $P‎_{3}‎$ on vertices $v_{i1},v‎_{i2}‎$ and $v_{i3}$ by adding two leaves $\ell‎‎^{1}_{ij}‎$ and $\ell‎‎^{2}_{ij}‎$ to $v_{ij}$, for $j=1,2,3$. Let the tree $T$ be the disjoint union of $T‎_{1},...,‎T‎_{k}$ by adding a path on the vertices $v‎_{12},...,v‎_{k2}‎‎$. Then $f:V(T)‎\rightarrow\{-1,1\}‎$ defined by,
$$f(v)=\left \{
\begin{array}{lll}
-1 & \mbox{if} & v=v‎_{ij}, ‎\ell‎‎^{2}_{i1}, \ell‎‎^{2}_{i3} ‎‎‎‎ \\
\ 1 & \mbox{if} & v‎\neq ‎v‎_{ij}, ‎\ell‎‎^{2}_{i1}, \ell‎‎^{2}_{i3}
\end{array}
\right.$$
is a maximum ISTDF of $T$ with weight $k$.
\end{proof}

We bound the ISTDN of a tree from below by considering its leaves and support vertices and characterize all trees attaining this bound. For this purpose, we introduce some notation. The set of leaves and support vertices of a tree $T$ are denoted by $L=L(T)$ and $S=S(T)$, respectively. Consider $L‎_{v}‎$ as the set of all leaves adjacent to the support vertex $v$, and $T'$ as the subgraph of $T$ induced by the set of support vertices.\vspace{1mm}\\
The following lemma will be useful.
\begin{lemma}\label{th9}
If $T$ is a tree, then there exists an ISTDF of $T$ of the weight $\gamma‎‎^{0}_{st}‎‎(T)$ that assigns to at least ‎$\lfloor‎\frac{\ell‎_{i}‎}{2}‎‎\rfloor$ leaves of the support vertex $v‎_{i}‎$ the value $1$, where $\ell‎_{i}‎$ is the number of leaves adjacent to $v‎_{i}‎$.
\end{lemma}
\begin{proof}
Let $f:V‎\rightarrow\{-1,1\}‎$ be an ISTDF of weight $\gamma‎‎^{0}_{st}‎‎(T)$. Suppose that there exists support vertex $v‎_{i}‎$ which is adjacent to at most $\lfloor‎\frac{\ell‎_{i}‎}{2}‎‎\rfloor-1$ vertices in $P$. Then, $f(N(v‎_{i}‎))‎\leq-deg(v‎_{i}‎)‎+\lfloor‎\frac{\ell‎_{i}‎}{2}‎‎\rfloor-1‎\leq-2‎$. Let $u$ be a leaf adjacent to $v‎_{i}‎$ with $f(u)=-1$. Define $f':V\rightarrow\{-1,1\}$ by
$$f'(v)=\left \{
\begin{array}{lll}
1 & \mbox{if} & v=u  \\
f(v) & \mbox{if} & v‎\neq‎‎ u.
\end{array}
\right.$$
Then $f'$ is an ISTDF of $T$ with weight $\gamma‎‎^{0}_{st}‎‎(T)+2$, which is a contradiction. Therefore, every support vertex $v‎_{i}‎$ is adjacent to at least $\lfloor‎\frac{\ell‎_{i}‎}{2}‎‎\rfloor$ vertices in $P$. Consider the support vertex $v‎_{1}‎$ and the vertices $u‎_{1}‎,...,u‎_{\lfloor‎\frac{\ell‎_{1}‎}{2}‎‎\rfloor}$ in $L‎_{v‎_{1}‎}‎$. Without loss of generality we can assume that $u‎_{1}‎,...,u‎_{\lfloor‎\frac{\ell‎_{1}‎}{2}‎‎\rfloor}$ have the value $1$ under $f$. Since $f$ is a maximum ISTDF then this statement holds for the other support vertices, as well.‎
\end{proof}
Now we define $‎\Omega‎‎$ to be the family of all trees $T$ satisfying:\vspace{1mm}\\
$(a)$ For any support vertex $w$ we have $|L‎_{w}‎|‎\geq2‎$ or $T$ is isomorphic to the path $P‎_{2}‎$;\\
$(b)$ $‎\Delta(T')‎‎\leq1‎$;\\
$(b_1)$ If $‎\Delta(T')‎‎=1‎$, then all of vertices of $T$ are leaves or support vertices and $|L‎_{w‎}|‎$ is even for all support vertex $w$;\\
$(b_2‎)$ if $‎\Delta(T')=0‎$, then $(i)$ $T$ is isomorphic to the star $K‎_{1,n-1}‎$ or $(ii)$ each support vertex is adjacent to just one vertex in $V‎\setminus(L‎\cup S‎)‎$, every vertex in $V‎\setminus(L‎\cup S‎)‎$ has at least one neighbor in $S$ and $|L‎_{w‎}|‎$ is even for all support vertices $w$.\vspace{1mm}\\
We are now in a position to present the following lower bound.
\begin{theorem}
Let $T$ be a tree of order $n‎$ with the set of support vertices $S=\{v‎_{1},...,v‎_{s}‎‎\}$. Then
$$\gamma‎‎^{0}_{st}‎‎(T)‎\geq-n+2(‎\lfloor‎\frac{\ell‎_{1}‎}{2}‎‎\rfloor‎‎+...+‎\lfloor‎\frac{\ell‎_{s}‎}{2}‎‎\rfloor)‎$$
where $\ell‎_{i}‎$ is the number of leaves adjacent to $v‎_{i}‎$. Moreover, the equality holds if and only if $T‎‎\in ‎\Omega‎‎$.
\end{theorem}
\begin{proof}
Let $f:V‎\rightarrow\{-1,1\}‎$ be a function which assigns $1$ to $‎\lfloor‎‎\frac{\ell‎_{i}‎}{2}‎\rfloor‎‎$ leaves of $v‎_{i}‎$ and $-1$ to all remaining vertices. Then, it is easy to see that $f$ is an ISTDF of $T$. Therefore
$$\gamma‎‎^{0}_{st}‎‎(T)‎\geq f(V)=-n+2(‎\lfloor‎\frac{\ell‎_{1}‎}{2}‎‎\rfloor‎‎+...+‎\lfloor‎\frac{\ell‎_{s}‎}{2}‎‎\rfloor).$$
Let $T$ be a tree for which the equality holds. So, we may assume that the above function $f$ is a maximum ISTDF of $T$. We first show that $T$ satisfies $(a)$. Without loss of generality we assume that $T$ is not isomorphic to the path $P‎_{2}‎$. Suppose that there exists a support vertex $v‎_{k}‎$ adjacent to just one leaf $u$. Thus, $s‎\geq2‎$. Then the function $g:V‎\rightarrow\{-1,1\}‎$ which assigns $1$ to $u$ and $‎\lfloor‎‎\frac{\ell‎_{i}‎}{2}‎\rfloor‎‎$ leaves of $v‎_{i}‎$, $i‎\neq k‎$, and $-1$ to all other vertices is an ISTDF of $T$ with weight $f(V)+1$, which is a contradiction. Therefore, all support vertices are adjacent to at least two leaves.\\
That the tree $T$ satisfies $(b)$, may be seen as follows. Let $‎\Delta(T')‎‎\geq2‎$. Then there are three vertices $v‎_{i-1}‎$, $v‎_{i}‎$ and $v‎_{i+1}‎$ on a path as a subgraph of $T'$. Let $u$ be a leaf adjacent to $v‎_{i}‎$ with $f(u)=-1$. Since $v‎_{i}‎$ is adjacent to the vertices $v‎_{i-1}‎$ and $v‎_{i+1}‎$ with $f(v‎_{i-1})=f(v‎_{i+1})=-1$, then $g:V‎\rightarrow\{-1,1\}$ defined by
$$g(v)=\left \{
\begin{array}{lll}
1 & \mbox{if} & v=u  \\
f(v) & \mbox{if} & v‎\neq‎‎ u
\end{array}
\right.$$
is an ISTDF of $T$ with weight $f(V)+2$, a contradiction. Thus, $‎\Delta(T')‎‎\leq‎1$. We now distinguish two cases depending on $‎\Delta(T')=0‎$ or $‎1‎$.\vspace{1mm}\\
{\bf Case 1.} $‎\Delta(T')=1‎$. Suppose to the contrary that there exists a vertex $v$ in $V‎\setminus (S‎\cup L)‎‎$ adjacent to a support vertex $v‎_{k}‎$ with degree one in $T'$. Then it must be assigned $-1$ under $f$. Let $u$ be a leaf adjacent to $v‎_{k}‎$ with $f(u)=-1$. We define a function $h:V‎\rightarrow\{-1,1\}$ by
$$h(v)=\left \{
\begin{array}{lll}
1 & \mbox{if} & v=u  \\
f(v) & \mbox{if} & v‎\neq‎‎ u.
\end{array}
\right.$$
Since $vv‎_{k}‎$ is an edge of $T$, then $h$ is an ISTDF of $T$ with weight $f(V)+2$, which is a contradiction. Therefore all of vertices of $T$ are leaves or support vertices. This implies that $S=\{v‎_{1},v‎_{2}‎‎\}$ and $v‎_{1}v‎_{2}$ is an edge of $T$. Now let $|L‎_{v‎_{1}‎}‎|$ be odd. Then the function $h':V‎\rightarrow\{-1,1\}‎$ that assigns $1$ to $\lfloor‎\frac{\ell‎_{1}‎}{2}‎‎\rfloor+1$ leaves of $v‎_{1}‎$, $\lfloor‎\frac{\ell‎_{2}‎}{2}‎‎\rfloor$ leaves of $v‎_{2}‎$ and $-1$ to all remaining vertices is an ISTDF of $T$ with weight $f(V)+2$, a contradiction. A similar argument shows that $|L‎_{v‎_{2}‎}‎|$ is even, as well.
\vspace{1mm}\\
{\bf Case 2.} $‎\Delta(T')=0‎$. If $T$ is isomorphic to the star $K‎_{1,n-1}‎$, then $(b_2)$ holds. Otherwise, $s‎\geq2‎$. Suppose to the contrary that there exists a support vertex $w$ adjacent to at least two vertices $u‎_{1}‎,u‎_{2}‎‎\in V‎\setminus(L‎\cup S‎)‎‎$ and $u$ is a leaf adjacent to $w$ with $f(u)=-1$. Since $f(u‎_{1})=f(u‎_{2})=-1$ then,  $r:V‎\rightarrow\{-1,1\}$ defined by
$$r(v)=\left \{
\begin{array}{lll}
1 & \mbox{if} & v=u  \\
f(v) & \mbox{if} & v‎\neq‎‎ u
\end{array}
\right.$$
is an ISTDF of $T$ with weight $f(V)+2$, which is a contradiction. Moreover, if $u$ is a vertex in $V‎\setminus(L‎\cup S‎)‎$ with no neighbor in $S$ then the function $r':V‎\rightarrow\{-1,1\}$ for which $r'(u)=1$ and $r'(v)=f(v)$ for all remaining vertices is an ISTDF of $T$ with weight $f(V)+2$, this contradicts the fact that $f$ is a maximum ISTDF of $T$. Finally, the proof of the fact that $|L‎_{w‎}|‎$ is even for all support vertex $w$ is similar to that of Case 1.\\
These two cases imply that $T$ satisfies $(b_1)$ and $(b_2)$.\\
Conversely, suppose that $T‎\in ‎\Omega‎‎$ and $f:V‎\rightarrow\{-1,1\}‎$ is an ISTDF of $T$ with weight
\begin{equation}
f(V)=\gamma‎‎^{0}_{st}‎‎(T)>-n+2(‎\lfloor‎\frac{\ell‎_{1}‎}{2}‎‎\rfloor‎‎+...+‎\lfloor‎\frac{\ell‎_{s}‎}{2}‎‎\rfloor).
\end{equation}
By Lemma \ref{th9}, we may assume that $f$ assigns $1$ to at least $\lfloor‎\frac{\ell‎_{i}‎}{2}‎‎\rfloor$ leaves $u‎‎^{i}_{1},...,u‎‎^{i}_{\lfloor‎\frac{\ell‎_{i}‎}{2}‎‎\rfloor}$ of the support vertex $v‎_{i}‎$, for $1‎\leq i‎\leq s‎‎$. The inequality (11) shows that there exists a vertex $u\in V‎\setminus‎ ‎\cup‎‎_{i=1}^{s}\{u‎‎^{i}_{1},...,u‎‎^{i}_{\lfloor‎\frac{\ell‎_{i}‎}{2}‎‎\rfloor}\}‎$ such that $f(u)=1$. Since all support vertices must be assigned $-1$ under $f$, then $u\in L‎_{v‎_{i}‎}‎$ or $u$ is a vertex in $V‎\setminus(L‎\cup S‎)‎$ adjacent to a support vertex $v‎_{i}‎$, for some $1‎\leq i‎\leq s‎‎$. It is not hard to see that this contradicts the fact that $T$ belongs to ‎$\Omega‎‎$ and $f(N(v‎_{i}‎))‎\leq0‎$, for all $1‎\leq i‎\leq s‎‎$. This completes the proof.
\end{proof}


\end{document}